\documentclass[envcountsect]{llncs}
\usepackage[utf8]{inputenc}
\usepackage{amssymb,amsmath,mathrsfs}
\usepackage[english]{babel}
\usepackage[shortlabels]{enumitem}
\usepackage{mathtools}
\usepackage[%
backend=biber,
bibencoding=utf8,
sorting=none,
style=numeric,
language=autobib,
autolang=other,
clearlang=true,
defernumbers=false,
sortcites=true,
doi=true,
isbn=true,
]{biblatex}
\usepackage[unicode,unicode=true,bookmarks=false]{hyperref}
\usepackage{xurl}
\hypersetup{breaklinks=true}
\usepackage{tikz}
\usetikzlibrary{automata, arrows.meta, positioning}
\usepackage{csquotes}

\usepackage{environ}

\newcommand{\repeattheorem}[1]{%
	\begingroup
	\renewcommand{\thetheorem}{\ref{#1}}%
	\expandafter\expandafter\expandafter\theorem
	\csname reptheorem@#1\endcsname
	\endtheorem
	\endgroup
}

\NewEnviron{reptheorem}[1]{%
	\global\expandafter\xdef\csname reptheorem@#1\endcsname{%
		\unexpanded\expandafter{\BODY}%
	}%
	\expandafter\theorem\BODY\unskip\label{#1}\endtheorem
}

\newcommand{\NN}{\mathbb{N}}
\newcommand{\ZZ}{\mathbb{Z}}
\newcommand{\QQ}{\mathbb{Q}}

\newcommand{\PrA}{\mathop{\mathbf{PrA}}\nolimits}

\newcommand{\BA}{\mathop{\mathbf{BA}}\nolimits}

\newcommand{\Th}{\mathop{\mathbf{Th}}\nolimits}

\newcommand{\Digit}{\mathop{Digit}\nolimits}

\makeatletter

\renewcommand{\epsilon}{\varepsilon}
\renewcommand{\phi}{\varphi}
\newcommand{\sref}[2]{\hyperref[#2]{#1 \ref*{#2}}}
\newcommand{\dref}[2]{\hyperref[#2]{ #1 }}

\newcommand{\Ac}{\mathcal{A}}

\newcommand{\Lc}{\mathcal{L}}

\newcommand{\eqdef}{\stackrel{\mbox{\tiny\rm def}}{=}}

\newcommand{\ra}{\rightarrow}
\newcommand{\Ra}{\Rightarrow}

\newcommand{\lra}{\leftrightarrow}

\spnewtheorem{hyp}{Conjecture}[section]{\bfseries}{\itshape}
\spnewtheorem{ex}{Example}{\bfseries}{\itshape}
\spnewtheorem{stm}{Statement}[section]{\bfseries}{\itshape}

\makeatletter
\newcommand{\dotminus}{\mathbin{\text{\@dotminus}}}

\newcommand{\@dotminus}{%
	\ooalign{\hidewidth\raise1ex\hbox{.}\hidewidth\cr$\m@th-$\cr}%
}
\makeatother

\begin{document}
	\author{Alexander Zapryagaev\thanks{The publication was prepared within the framework of the Academic Fund Program at HSE University (grant 23-00-022).}}
	\title{Some properties of B\"uchi Arithmetics}
	\institute{National Research University Higher School of Economics, 6, Usacheva Str., Moscow, 119333, Russian Federation}
	
	\maketitle
	
	\begin{abstract}
		B\"uchi arithmetics $\BA_n$, $n\ge 2$, are extensions of Presburger arithmetic with an unary functional symbol $V_n(x)$ denoting the largest power of $n$ that divides $x$. A rank of a linear order is the minimal number of condensations required to reach a finite order. We show that linear orders of arbitrarily large finite rank can be interpreted in $\BA_n$. We also prove that the extension of the axioms of Presburger arithmetic with the inductive definition of $V_n$ does not yield an axiomatization of $\BA_n$.
	\end{abstract}
	
	\section{Preliminaries}
	
	\begin{definition}
		A \textbf{B\"uchi arithmetic} $\BA_n$, $n\ge 2$, is the theory $\Th(\NN;=,+,V_n)$ where $V_n$ is an unary functional symbol such that $V_n(x)$ is the largest power of $n$ that divides $x$ (we set $V_n(0):=0$ by definition).
	\end{definition}
	
	These theories were proposed by R.~B\"uchi in order to describe the recognizability of sets of natural numbers by finite automata through definability in some arithmetic language.
	
	Let $\Digit_n(x,y)$ be the digit corresponding to $n^y$ in the $n$-ary expansion of $x\in\NN$. Consider an automaton over the alphabet $\{0,\ldots,n-1\}^m$ that, at step $k$, receives the input $(\Digit_n(x_1,k),\ldots,\Digit_n(x_m,k))$ of the digits corresponding to $n^k$ in the $n$-ary expansion of $(x_1,\ldots,x_m)$. We say the automaton accepts the tuple $(x_1,\ldots,x_m)$ if it accepts the sequence of tuples $(\Digit_n(x_1,k),\ldots,\Digit_n(x_m,k))$.
	
	Under the conditions above, the following classic result by V\'eronique Bruy\`ere \cite{bruyere,bv} holds:
	
	\begin{theorem}\label{there}
		Let $\varphi(x_1,\ldots,x_m)$ be a $\BA_n$-formula. Then there is an effectively constructed automaton $\Ac$ such that $(a_1,\ldots,a_m)$ is accepted by $\Ac$ iff $\NN\models\varphi(a_1,\ldots,a_m)$.
		
		Contrariwise, let $\Ac$ be a finite automaton working on $m$-tuples of $n$-ary natural numbers. Then there is an effectively constructed $\BA_n$-formula $\varphi(x_1,\ldots,x_m)$ such that $\NN\models\varphi(a_1,\ldots,a_m)$ iff $(a_1,\ldots,a_m)$ is accepted by $\Ac$. Furthermore, this formula is of complexity class not surpassing $\Sigma_2$ \cite{haase}.
	\end{theorem}

	\begin{definition}
		Let $\mathfrak{B}$ be a first order structure with the language containing equality and predicate symbols $P_1,\ldots,P_n$. $\mathfrak{B}$ is called \emph{automatic} \cite[Definition 1.4]{kn} if there a language $\Lc\subseteq\Omega^*$ over a finite alphabet $\Omega$ and a 
		surjective mapping $c\colon\Lc\ra\mathfrak{B}$ such that the following sets are recognizable by some automaton over $\Omega$ ($\overline{x}_i\in\Omega^*$):
		
		\begin{enumerate}
			\item The language $\Lc$;
			\item The set of all pairs $(\overline{x},\overline{y})\in\Lc^2$ that $c(\overline{x})=c(\overline{y})$;
			\item The set of all tuples $(\overline{x}_1,\ldots,\overline{x}_n)\in\Lc^n$ that $P_i(c(\overline{x}_1),\ldots,c(\overline{x}_n))$ holds in $\mathfrak{B}$ for each predicate symbol $P_i$ in the language of $\mathfrak{B}$.
		\end{enumerate}
	\end{definition}
	
	By applying \sref{Statement}{there}, we may represent each of the automata in the definition above by a corresponding formula of some B\"uchi arithmetic $\BA_n$ (the finite alphabet $\Omega$ may be, without loss of generality, taken to be $\{0,1,\ldots,n-1\}$). Taken together, these formulas provide a translation from the language of $\mathfrak{B}$ into $(=,+,V_n)$ such that $\mathfrak{B}$ is isomorphic to an internal model obtained by this translation in $\NN$. In other words, automatic structures are exactly those on-dimensionally, not necessarily with absolute equality, that are interpretable in $(\NN;=,+,V_n)$.
	
	As shown in the author's dissertation \cite{zapthe} (Theorem 4.3.4), the theories $\BA_n$ are mutually interpretable for distinct $n\ge 2$:
	
	\begin{theorem}
		Each $\BA_k$ is interpretable in any of $\BA_l$, $k,l\ge 2$.
	\end{theorem}
	
	This is shown by the combination of two following theorems:
	
	\begin{theorem}
		Each $\BA_{k^2}$ can be interpreted in $\BA_k$.
	\end{theorem}
	
	\begin{theorem}
		Each $\BA_{k}$ can be interpreted in $\BA_{k+1}$, $k\ge 2$.
	\end{theorem}
	
	\section{Linear orders in B\"uchi arithmetics}
	
	Albert~Visser asked the question: \emph{for which arithmetical theories $T$ all their interpretations in themselves are provably isomorphic to the trivial one?}
	
	In \cite{pz} and \cite{zap2022}, the author has established:
	
	\begin{theorem}\label{one}
		\begin{enumerate}
			\item Let $\iota$ be a (one-dimensional or multi-dimensional) interpretation of $\PrA$ in $(\NN;=,+)$. The the internal model induced by $\iota$ is always isomorphic to the standard one.
			\item This isomorphism can always be expressed by a formula in the language of $\PrA$.
		\end{enumerate}
	\end{theorem}
	
	\begin{theorem}
		Let $\iota$ be a (one-dimensional or multi-dimensional) interpretation of $\BA_n$ in $(\NN;=,+,V_n)$. The the internal model induced by $\iota$ is always isomorphic to the standard one.
	\end{theorem}
	
	The result of \sref{Theorem}{one} $(1)$ was established by studying the linear orders interpretable in $\PrA$, obtaining a necessary condition based on the notion of \textbf{$VD^*$-rank} \cite{krs}.
	
	\begin{definition}
		Let $(L,<)$ be a linear order. By transfinite recursion, we introduce a family of equivalence relations $\simeq_{\alpha}$, $\alpha\in\mathbf{Ord}$ on $L$:
		
		\begin{enumerate}
			\item $\simeq_0$ is equality;
			\item $a\simeq_{\alpha+1}b$, if $|\{c\in L\mid (a<c<b)\mbox{ or }(b<c<a)\}/{\simeq_{\alpha}}|$ is finite;
			\item $\simeq_{\lambda}=\bigcup\limits_{\beta<\lambda}\simeq_{\alpha}$ when $\lambda$ is a limit ordinal.
		\end{enumerate}
		
		A \textbf{rank} $\mathrm{rk}(L,<)\in\mathbf{Ord}\cup\{\infty\}$ of the order $(L,<)$ is the smallest $\alpha$ such that $L/{\simeq_{\alpha}}$ is finite or $\infty$ if such does not exist.
	\end{definition}
	
	It is known \cite{rosenstein} that the \emph{scattered} linear orders, that is, not containing a suborder isomorphic to $\QQ$, exactly coincide with the orders of rank below $\infty$.
	
	The following condition has been established in \cite{pz}:
	
	\begin{theorem}
		All linear orders that are $m$-dimensionally interpretable in the structure $(\NN;=,+)$ have rank $\le m$.
	\end{theorem}
	
	As $\NN+\ZZ\cdot\QQ$ is not even scattered, a non-standard model $\PrA$ cannot be interpreted in $(\NN;=,+)$. In fact, the following complete criterion was very recently reached:
	
	\begin{theorem}[\cite{orders}]
		A linear order $(L,<)$ is $m$-dimensionally interpretable in $(\NN;=,+)$ for some $m\ge 1$ iff there exists some $k\in N$ and a $\PrA$-definable set $D\in\ZZ^k$ such that $L$ is isomorphic to the restriction of the lexicographic ordering on $\ZZ^k$ onto $D$.
	\end{theorem}
	
	Yet, the same rank condition is not extended to $\BA_n$. The statement holds:
	
	\begin{theorem}
		For each $n$, there is an order of rank $n$ interpretable in $\BA_2$.
	\end{theorem}
	\begin{proof}
		The original order on $\NN$ is an order of rank $1$: $$x\le_1 y:= x\le y.$$
		
		The following order is of rank $n=2$: $$x\le_2 y:= V_2(x)<V_2(y)\vee V_2(x)=V_2(y)\wedge(x\le y).$$
		
		Indeed, for each $k>0$, there are infinitely many natural numbers of the form $m={2^k}\cdot t$, $t$ is odd, for which $V_2(m)=2^k$. Hence, the order $\le_2$ is isomorphic to $1+\NN+\NN+\ldots\cong\NN\times\NN$.
		
		Now we consider a particular subsequence $m_t={2^k}\cdot t$, for all odd $t$, of all the numbers with $V_2$ equal to $2^k$. Note that the sequence $m_t-V_2(m_t)={2^k}\cdot t-2^k=2^k(t-1)=2^{k+1}\cdot\frac{t-1}{2}$, $\frac{t-1}{2}$ spans all of $\NN$. Hence, all values of $V_2$, starting from $2^{k+1}$, occur in the sequence $V_2(m_t-V_2(m_t))$ infinitely many times (and $0$ occurs once). Thus, the order \begin{multline*}x\le_3 y:= V_2(x)<V_2(y)\vee V_2(x)=V_2(y)\wedge V_2(x-V_2(x))<V_2(y-V_2(y))\vee\\V_2(x)=V_2(y)\wedge V_2(x-V_2(x))=V_2(y-V_2(y))\wedge x\le y\end{multline*}
		\noindent
		is a definable order isomorphic to $1+\NN\times(1+\NN\times\NN)\cong\NN^3$, of rank $3$.
		
		Iterating this process (considering $V_2(x-V_2(x)-V_2(x-V_2(x)))$ in the case of $n=4$, and so on), we obtain the required definable orders of however large finite rank.
	\end{proof}
	
	\section{Towards the axiomatization of $\BA_n$}
	
	It is well-known that $\PrA$ in the extended language $\{=,+,<,0,1,\{\equiv_n\}_{n=2}^\infty\}$ has an equivalent axiomatic definition, as the first-order theory given by the following recursive set of axioms ($\underline{n}\eqdef1+\ldots+1$ $n$~times):
	
	\begin{enumerate}
		\item $x=0\lra\forall y\:(x+y=y)$
		\item $x<y\lra\exists z\:((x+z=y)\wedge\neg(z=0))$
		\item $x=1\lra 0<x\wedge\neg\exists z\:(0<z\wedge z<x)$
		\item $x\equiv_n y\lra\exists u\:(x=\underline{n}u+y\vee y=\underline{n}u+x)$
		\item $\neg(x+1=0)$
		\item $x+z=y+z\ra x=y$
		\item $(x+y)+z=x+(y+z)$
		\item $x=0\vee\exists y\:(x=y+1)$
		\item $x+y=y+x$
		\item $x<y\vee x=y\vee y<x$
		\item $(x\equiv_n 0)\vee(x\equiv_n 1)\vee\ldots\vee(x\equiv_n\underline{n-1})$ 
	\end{enumerate}
	
	Here $\underline{n}$ abbreviates $1+\ldots+1$ $n$ times. As proven by Presburger \cite{presburger}, in this extended language $\PrA$ admits quantifier elimination.
	
	It may seem natural that an axiomatization of $\BA_2$ could be created based on these axioms, by extending them with the inductive definition of $V_2$:
	
	\begin{enumerate}
		\setcounter{enumi}{11}
		\item $V_2(x)=0\lra x=0$
		\item $\neg\exists t\:(t+t=x)\ra V_2(x)=1$
		\item $\exists t\:(t+t=x)\ra V_2(x)=V_2(t)+V_2(t)$
	\end{enumerate}
	
	However, as the following result shows, this is not a sufficient axiomatization.
	
	\begin{theorem}
		The axioms and axiom schemes $(1)$--$(14)$ do not form an axiomatization of $\BA_2$.
	\end{theorem}
	\begin{proof}
		We shall construct an explicit structure that models the axioms $(1)$--$(14)$ and then provide a statement true in $\BA_2$ that does not hold in it.
		
		We start with all tuples $$\{(p,q)\mid p\in\QQ\ge 0,q\in\ZZ,p=0\Ra q\in\NN\}$$ and addition defined componentwise. This structure already fulfills all the axioms of $\PrA$, including the possibility of division by any standard natural $n$ (axiom schema $(11)$). Clearly, the elements $\{(0,n)\mid n\in\NN\}$ play the role of standard natural numbers $n$. All the remaining elements will be henceforth called \emph{non-standard numbers}.
		
		Next, we define $V_2(p,q)$ as follows. First, the standard natural numbers receive the expected value: $V_2(0,n):=(0,V_2(n))$.
		
		Each odd non-standard number (that is, such $(p,q)$ that $q$ is odd) will have $V_2(p,q):=(0,1)$, corresponding to the fact they cannot be divided by two.
		
		Each even non-standard number $(p,q)$ such that $q\neq 0$ follows $V_2(p,q):=(0,V_2(q))$, as such numbers can be divided by $2$ exactly $V_2(q)$ times before reaching the number $(p/2^{V_2(q)},q/2^{V_2(q)})$ that is odd. (Note that $p$ are positive rational numbers that can be divided an unlimited amount of times.)
		
		Finally, for the non-standard numbers of the form $(p,0)$ that can be infinitely divided by two, we set their value of $V_2$ to be equal to themselves: $V_2(p,0):=(p,0)$. Unlike the other cases, these values of $V_2$ are non-standard themselves.
		
		Clearly, axioms $(12)--(14)$ hold in this structure. However, the following formula is true in $(\NN;=,+,V_2)$ and thus is a theorem of $\BA_2$: $$\forall x\:(V_2(x)=x\ra\forall y\:(x<y<x+x\ra V_2(y)<y)).$$
		
		It expresses the idea "between $2^k$ and $2^{k+1}$, there are no more powers of $2$". Obviously, it holds in $(\NN;=,+,V_2)$. However, between each two non-standard "powers of $2$" $(p,0)$ and $(2p,0)$ there is infinitely many additional elements with the property $V_2(n)=n$, such as $(3p/2,0)$.
	\end{proof}
	
	\printbibliography[heading=bibintoc,title={References}]
\end{document}